\documentclass[11pt,oneside]{amsart}
\usepackage{amssymb,amsmath,amsfonts}  
\usepackage{verbatim,a4wide}
\usepackage[cp1250]{inputenc}
\usepackage[OT4]{fontenc}
\usepackage[english]{babel}

\usepackage{bm}
\numberwithin{equation}{section}
\newtheorem{thm}{Theorem}[section]

\newtheorem{alg}[thm]{Algorithm}

\newtheorem{rem}[thm]{\it Remark }

\newcommand{\sL}{\mbox{$\mathcal L$}}
\newcommand{\dx}{\textrm{\,d}x}
\newcommand{\N}{\mathbb N}

\newcommand{\Z}{\mathbb Z}
\newcommand{\Zplus}{\Z_+}
\newcommand{\C}{\mathbb C}

\newcommand{\ip}[2]{\left\langle #1,\,#2\right\rangle}
\newcommand{\bigO}[1]{\mbox{${\mathcal O}(#1)$}}

\newcommand{\hyper}[5]{\,{}_{#1}F_{#2}\!\left(%
            \begin{array}{cc}{\displaystyle{#3}}\\[0.25ex]%
            {\displaystyle{#4}} \end{array}\bigg|\,{\displaystyle{#5}}%
            \right)}

\newcommand{\rbinom}[2]{\mbox{$\displaystyle\binom{#1}{#2}^{\!\!-1}\!\!$}}

\begin{document}
\title{B\'ezier representation of the constrained dual Bernstein polynomials}
\author{Stanis{\l}aw Lewanowicz${}^\star$}
\thanks{${}^\star$Corresponding author}

\author{Pawe{\l} Wo\'zny}

\address{Institute of Computer Science, University of Wroc{\l}aw,
         ul.~F.~Joliot-Curie 15, 50-383 Wroc{\l}aw, Poland}
\email{\{Stanislaw.Lewanowicz, Pawel.Wozny\}@ii.uni.wroc.pl}
\date{\today}

\maketitle
\thispagestyle{empty}

\begin{abstract} \small 
Explicit formulae  for the B\'ezier coefficients of the constrained 
dual Bernstein basis polynomials
are derived in terms of the Hahn orthogonal polynomials. 
Using difference properties of the latter
polynomials, efficient recursive scheme is obtained to compute these
coefficients. Applications of this result
to some problems of CAGD is  discussed.\\


\end{abstract}


\section{Introduction}
                                                        \label{SS:Intr}

Let $\Pi_n^{(k,l)}$, where $k+l\le n$, be the space of all polynomials of degree $\le n$,
whose derivatives
of order  $\le k-1$ at $t=0$, as well as derivatives of order  $\le l-1$ at $t=1$,
vanish. In particular,  $\Pi_n:=\Pi_n^{(0,0)}$ is the space of the unconstrained polynomials
of degree at most $n$. Obviously,  $\mbox{dim}\;\Pi_n^{(k,l)}=n-k-l+1$, and the Bernstein
polynomials
\(
\left\{B^n_k,B^n_{k+1},\ldots,B^n_{n-l}\right\},
\)
where
\[
   B^n_{i}(x)=\binom ni x^i(1-x)^{n-i} \qquad (0\le i\le n),
\]
form a basis of this space.
There is a unique \textit{dual constrained Bernstein basis of degree} $n$,
\[ 
\left\{D^{(n,k,l)}_k(x;\alpha,\beta),\,D^{(n,k,l)}_{k+1}(x;\alpha,\beta),\ldots,
D^{(n,k,l)}_{n-l}(x;\alpha,\beta)\right\}\subset\Pi_n^{(k,l)},
\] 
satisfying the relation
\(
	\ip{D^{(n,k,l)}_i}{B^n_j}=\delta_{ij}\quad (i,j=k,k+1,\ldots,n-l),
\)
where $\delta_{ij}$ equals 1 if $i=j$ and 0 otherwise, and the inner product $\langle\cdot,\cdot\rangle$ is given by
\begin{equation}\label{E:Jinprod}
\ip fg := \int_{0}^{1}
\,(1-x)^{\alpha}x^{\beta}\,f(x)\,g(x)\,\dx \qquad (\alpha,\;\beta>-1).
\end{equation}
The unconstrained dual basis $D^n_i(x;\alpha,\beta)$
corresponds to the case $k=l=0$, hence $D^n_i(x;\alpha,\beta)=D^{(n,0,0)}_i(x;\alpha,\beta)$.

Dual Bernstein polynomials $D^{(n,k,l)}_i(x;\alpha,\beta)$ were studied by Ciesielski \cite{Cie87b}
(case  $\alpha=\beta=0$ and $k=l=0$), also by J\"uttler \cite{Jue98} 
(case $\alpha=\beta=0$ and general $k=l$), Rababah and Al-Natour \cite{RA07} 
($k=l=0$ and general $\alpha,\,\beta$) and \cite{RA08} (general $\alpha,\,\beta$ and $k=l$). 
  
In \cite{LW06}, we have given
  explicit expressions for $D^n_i(x;\alpha,\beta)$  in terms of the shifted Jacobi 
  polynomials $R^{(\alpha,\beta)}_{j}(x)$, which, as it is well known, 
form a set of orthogonal polynomials with respect to the inner product \eqref{E:Jinprod}. 
See also \eqref{E:B-J}. 
In \cite{WL09}, we show that the constrained polynomials $D^{(n,k,l)}_i(x;\alpha,\beta)$ are
related to the unconstrained dual Bernstein polynomials of degree $n-k-l$,
with parameters $\alpha+2l$ and $\beta+2k$.  Consequently, using the above-mentioned 
results of \cite{LW06}, we obtain an
 explicit  orthogonal  representation, as well as  degree elevation formula 
 for the constrained dual Bernstein basis polynomials. The latter result plays
 a crucial role in the 
 algorithm of multi-degree reduction of B\'ezier curves, given in the cited paper.

In some  computer-aided geometric design applications, like curve intersection using
B\'ezier clipping 
(see, e.g., \cite{SN90,BJ07,LZLW09}), 
polynomial approximation of rational B\'ezier curves \cite{LWK10},
or degree reduction of B\'ezier curves (see, e.g., \cite{WL09} and the papers cited therein),
it is desirable to have the
B\'ezier-Bernstein representation of the dual constrained polynomials,
\[
	 D^{(n,k,l)}_i(x;\alpha,\beta)=\sum_{j=k}^{n-l}C_{ij}(n,k,l,\alpha,\beta)\,  
	                   B^{n}_j(x).
\]
In this paper, we give an efficient algorithm to compute the coefficients
$C_{ij}(n,k,l,\alpha,\beta)$ for $i,j=1,2\ldots,n$,
which has $\bigO{n^2}$ complexity. This should be compared to a rather complicated
recursive algorithm 
given in \cite{Jue98} and its generalization of \cite{RA08}.

In the sequel, we use the notation 
\[
	(c)_k:=\prod_{j=0}^{k-1}(c+j)\quad (k\ge0).
\]
The \textit{generalized hypergeometric series} is defined by
(see, e.g., \cite[\S2.1]{AAR99}, or \cite[\S1.4]{KLS10})
\[
 \hyper rs {a_1,\ldots, a_r}{b_1,\ldots,b_s}z
 :=\sum_{k=0}^{\infty}\frac{(a_1)_k\ldots(a_r)_k}
                           {k!(b_1)_k\ldots(b_s)_k}\,z^k,
\]
where $r,\,s\in\Zplus$ and $a_1,\ldots, a_r,\,b_1,\ldots,\,b_s$, $z$ $\in \C$.

\section{Dual Bernstein polynomials}
\label{S:dBer}
 \textit{Bernstein polynomial basis of  degree} $n$
($n\in\N$) is given by
\[
B_{i}^{n}(x):=\binom{n}{i}\, x^{i}\,(1-x)^{n-i}\qquad  (0 \leq i\leq n).
\]
Recall that
\begin{equation}\label{E:SJ}
R^{(\alpha,\beta)}_{m}(x):=  \frac{(\alpha+1)_m}{m!}
	\hyper{2}{1}{-m,\,m+\alpha+\beta+1}{\alpha+1}{1-x}
\end{equation}
are the \textit{shifted Jacobi polynomials}~\cite[p. 280, Eq. (46)]{Luk69}.
For $\alpha>-1$ and $\beta>-1$, polynomials \eqref{E:SJ} are
orthogonal with respect to the scalar product \eqref{E:Jinprod}
(see, e.g., \cite[p. 273]{Luk69}); more specifically,
\[ 
\ip{ R_{m}^{(\alpha,\beta)}}{R_{n}^{(\alpha,\beta)}}
=\delta_{mn}\,B(\alpha+1,\,\beta+1)\,
       \frac{(\alpha+1)_m(\beta+1)_m}{m!(2m/\sigma+1)(\sigma)_{m}}
       \qquad (m,n\ge0)
\] 
with $\sigma:=\alpha+\beta+1$ and
\(
B(\lambda,\,\mu)={\Gamma(\lambda)\Gamma(\mu)}/{\Gamma(\lambda+\mu)}.
\)
Recall that the \textit{Hahn polynomials} are defined by  \cite[\S9.5]{KLS10}
\begin{equation}\label{E:Hahn}
  Q_m\!\left(x;\alpha,\beta,N\right):=
    \hyper32{-m,m+\alpha+\beta+1,-x}{\alpha+1,-N}1
    \quad (m=0,1,\ldots ,N;\;N\in\N).
\end{equation} 
Bernstein polynomials have the following representation
in terms of shifted Jacobi polynomial basis:
\begin{align}\label{E:B-J}
 B^n_i(x)
 =&\,\binom{n}i
  {(\alpha+1)_{n-i}(\beta+1)_i}  
  \sum_{j=0}^{n}\,
   \frac{(2j+\sigma)(-n)_j}
        {(\alpha+1)_j(j+\sigma)_{n+1}}
   \,Q_{j}\!\left(i;\beta,\alpha,n\right)
           R^{(\alpha,\beta)}_j(x), 
\end{align}
where $0\le i\le n$. This result follows from a formula given
in \cite{RZAG98}, using \cite[Cor. 3.3.5]{AAR99}; see also \cite[Eq. (5.4)]{LW06}.
Conversely, shifted Jacobi polynomials have
the following representation
in terms of Bernstein basis \cite{Cie87a}:
\begin{align}\label{E:J-B}
R^{(\alpha,\beta)}_i(x)=&
     \dfrac{(\alpha+1)_i}{i!}\,
      \,\sum_{j=0}^{n}\,Q_i\left(n-j;\alpha,\beta,n\right)
            \,B^n_j(x) \qquad (0\le i\le n).
\end{align}

Associated with the basis $B^n_i(x)$ there is a unique 
\textit{dual  Bernstein polynomial  basis of degree $n$},
\[
D^n_0(x;\alpha,\beta),\,D^n_1(x;\alpha,\beta),\ldots ,D^n_n(x;\alpha,\beta),
\]
defined so that
\begin{equation}\label{E:J-biort}
 \ip{D^n_i}{B^n_j}=\delta_{ij}\qquad (i,j=0,1,\ldots ,n).
\end{equation}
In \cite{LW06}, we gave a number of properties of polynomials $D^n_i(x;\alpha,\beta)$,
including 
the expansion
\begin{align}
\label{E:D-J}
D^n_i(x;\alpha,\beta) = &\frac1{B(\alpha+1,\,\beta+1)}
         \sum_{j=0}^{n}(-1)^j
       \frac{(2j/\sigma+1)(\sigma)_{j}}{(\alpha+1)_j}
              Q_j(i;\beta,\alpha,n)R_{j}^{(\alpha,\beta)}(x),
\end{align}
where $0\le i\le n$, as well as the following alternative formula 
representing the dual Bernstein polynomial $D^n_i(x;\alpha,\beta)$
($0\le i\le n$) as a "short" linear combination of 
 Jacobi polynomials with shifted parameters:
\begin{equation}\label{E:dB-nicef}
D^n_i(x;\alpha,\beta)=\frac{(-1)^{n-i}(\sigma+1)_{n}}
                           {B(\alpha+1,\,\beta+1)\,(\alpha+1)_{n-i}(\beta+1)_i}
\sum_{j=0}^{i}\frac{(-i)_j}{(-n)_j}  
\,R^{(\alpha,\beta+j+1)}_{n-j}(x).
\end{equation}

In CAGD applications, it is desirable to express the dual basis polynomials in terms of 
the Bernstein basis. We have the following result.
\begin{thm}
		\label{T:DinB}
For $i=0,1,\ldots,n$, we have 		
\begin{equation}\label{E:DinB}
D^n_i(x;\alpha,\beta)=
                      \sum_{j=0}^{n}c_{ij}(n,\alpha,\beta)\,B^n_j(x),
\end{equation}
where the B\'ezier coefficients $c_{ij}(n,\alpha,\beta)$ ($0\le j\le n$)
are given by any of the following two forms:
\begin{align}\label{E:cij}
c_{ij}(n,\alpha,\beta):=& 
                           \frac{1}{B(\alpha+1,\,\beta+1)}
			   \sum_{m=0}^n\frac{(2m/\sigma+1)(\beta+1)_m(\sigma)_m}{m!(\alpha+1)_m}
			         Q_m(i;\beta,\alpha,n)\,Q_m(j;\beta,\alpha,n),\\[1ex]
				 \label{E:cij-sh}
c_{ij}(n,\alpha,\beta):=& \frac{(-1)^{n}(\sigma+1)_{n}(-\alpha-n)_i}
                           {B(\alpha+1,\,\beta+1)\,n!(\beta+1)_i}\sum_{h=0}^{i}\frac{(-i)_h}{(-\alpha-n)_h}
			   Q_{n-h}(n-j;\alpha,\beta+h+1,n). 				 
\end{align}
\end{thm}
\begin{proof}
Formula \eqref{E:DinB} with $c_{ij}(n,\alpha,\beta)$ given by \eqref{E:cij}  
(respectively, \eqref{E:cij-sh})  follows
by inserting expansion  \eqref{E:J-B} in  \eqref{E:D-J} (respectively, \eqref{E:dB-nicef})
and doing some algebra.
\end{proof}

Now, we show that the coefficients $c_{ij}(n,\alpha,\beta)$ ($i,j=1,2\ldots,n$) given in
Theorem~\ref{T:DinB}
can be computed recursively, with a cost  $\bigO{n^2}$.
\begin{thm}
	\label{T:c-rec}
Quantities $c_{ij}\equiv c_{ij}(n,\alpha,\beta)$ defined in \eqref{E:cij}, \eqref{E:cij-sh}
satisfy the following recurrence relation: 
	\begin{align}\label{E:c-rec}
c_{i+1,j}=&\frac1{A(i)}\Big\{(i-j)(2i+2j-2n-\alpha+\beta)\,c_{ij}\nonumber \\
   &\hphantom{\frac1{A(i)}\Big\{}
   + B(j)\,c_{i,j-1}+A(j)\,c_{i,j+1}-B(i)\,c_{i-1,j}\Big\},
	 \end{align}
where $0\le i\le n-1$, $0\le j\le n$ and
\[  
A(h):=(h-n)(h+\beta+1),\qquad
B(h):=h(h-n-\alpha-1). 
\] 
The starting values are
\begin{equation}
	\label{E:crec-start}
c_{0j}
 =
 \frac{(\sigma+1)_n(\beta+2)_n}{n!\,B(\alpha+1,\beta+1)}\cdot
 \frac{(-1)^j}{(\alpha+1)_{n-j}(\beta+2)_j}
 \qquad (0\le j\le n). 	
\end{equation}
\end{thm}
\begin{proof} Observe that the expression on the right-hand side of \eqref{E:cij} 
is the linear combination of products $Q_m(i;\beta,\alpha,n)\,Q_m(j;\beta,\alpha,n)$,
with the coefficients independent of $i$ and $j$. 
Hence,  the proof of the Theorem can  be based on an idea similar to the one
used in  \cite[Thm A6]{WL09} to prove a recursion satisfied by the dual discrete Bernstein
polynomials.

First, recall that the Hahn polynomials satisfy the difference equation (see, e.g., {\cite[\S9.5]{KLS10}}),
\[  
	\sL^N_x Q_m(x;\beta,\alpha,N)=m(m+\sigma)\,Q_m(x;\beta,\alpha,N),
\]
where the operator $\sL^N_x$ is given by
\[ 
\sL^N_x\,y(x)=a_N(x)\,y(x+1)-c_N(x)\,y(x)+b_N(x)\,y(x-1)	
\]  
with
\begin{gather*}
a_N(x):=(x-N)(x+\beta+1),\qquad
b_N(x):=x(x-\alpha-N-1),\\[1ex]
c_N(x):=a_N(x)+b_N(x).	
\end{gather*}
Now,  using equation \eqref{E:cij} we easily obtain that
\begin{align*}
	\sL^n_i\, c_{ij} = &\,
 \sum_{m=0}^n\frac{(2m/\sigma+1)(\beta+1)_m(\sigma)_m}{B(\alpha+1,\,\beta+1)m!(\alpha+1)_m} \\
		&\hphantom{\sum_{m=0}^nn} 
		\times m(m+\sigma)\,Q_m(i;\beta,\alpha,n)\,Q_m(j;\beta,\alpha,n)\\
				 = &\,\sL^n_j\, c_{ij}.
\end{align*}
This implies equation \eqref{E:c-rec}.

By \eqref{E:cij-sh} and \eqref{E:Hahn}, 
the sum in $c_{0j}$ reduces to  ${}_2F_1(j-n,n+\sigma+1;\alpha+1;1)$
which can be summed by the Chu-Vandermonde formula (see, e.g., \cite[(1.5.4)]{KLS10}).
Hence follows \eqref{E:crec-start}.
\end{proof}
\begin{rem}
 Rababah and Al-Natour \cite{RA08} have given the
formula
\begin{align*}
c_{ij}(n,\alpha,\beta) =& {(-1)^{i+j}}\rbinom{n}{i}\rbinom{n}{j}\\[1ex]
&\times\sum_{h=0}^{\min(i,j)}\!\!(2h+\beta+1)\binom{n+\sigma+h}{n+\beta+h+1}
\rbinom{n+\alpha-h}{n-h}\, v_{ih}\, v_{jh}	  
\end{align*}
with 
\[
	 v_{mh}:=\binom{n+\beta+h+1}{n-m}\binom{n+\alpha-h}{n+\alpha-m},
\]
which contains the earlier result of J\"uttler \cite{Jue98} for $\alpha=\beta=0$.
\end{rem}

\section{Constrained dual Bernstein polynomials}
							\label{S:constrdual}

Let $\Pi_n^{(k,l)}$, where $k+l\le n$, be the space of all polynomials of degree $\le n$,
whose derivatives
of order  $\le k-1$ at $t=0$, as well as derivatives of order  $\le l-1$ at $t=1$,
vanish:
\[ 
\Pi_n^{(k,l)}:=
       \left\{P\in\Pi_n\::\:
       P^{(i)}(0)=0\;\:(0\le i\le k-1);
       \;\:        
       P^{(j)}(1)=0\;\: (0\le j\le l-1)\right\}.
\]  
  In particular,  $\Pi_n:=\Pi_n^{(0,0)}$ is the space of the unconstrained polynomials
of degree at most $n$. Obviously,  $\mbox{dim}\;\Pi_n^{(k,l)}=n-k-l+1$, and the Bernstein polynomials
\(
\left\{B^n_k,B^n_{k+1},\ldots,B^n_{n-l}\right\}
\)
form a basis of this space.
There is a unique \textit{dual constrained Bernstein basis of degree} $n$,
\[
D^{(n,k,l)}_k(x;\alpha,\beta),\,D^{(n,k,l)}_{k+1}(x;\alpha,\beta),\ldots,
D^{(n,k,l)}_{n-l}(x;\alpha,\beta),
\]
satisfying
\[
	\ip{D^{(n,k,l)}_i}{B^n_j} =\delta_{ij}\qquad (i,j=k,k+1,\ldots,n-l).
\]
Recently, we have shown that the  polynomials $D^{(n,k,l)}_i$ can be expressed in terms of  
dual Bernstein polynomials without constraints. Namely, we have the following result.  
\begin{thm}[\cite{WL09}] \label{T:contrDinD}
For $i=k,k+1,\ldots,n-l$, the following formula holds:
\begin{equation}\label{E:constrD}
D^{(n,k,l)}_i(x;\alpha,\beta)
= U_i\cdot x^k(1-x)^l\, D^{n-k-l}_{i-k}(x;\alpha+2l,\beta+2k), 
\end{equation}
where
\begin{equation}
	\label{E:U}
	U_i:=\binom{n-k-l}{i-k}\rbinom{n}{i}.
\end{equation}
\end{thm}
Now, using expansion \eqref{E:DinB} in \eqref{E:constrD}, 
 the following result is easily obtained.
\begin{thm} \label{T:constrDinB}
The constrained dual basis polynomials have the  B\'ezier-Bernstein representation
\begin{equation}
	\label{E:constrDinB}
	 D^{(n,k,l)}_i(x;\alpha,\beta)=\sum_{j=k}^{n-l}C_{ij}(n,k,l,\alpha,\beta)\,  
	                   B^{n}_j(x)\qquad (k\le i\le n-l),
\end{equation}
where 
\begin{equation}
	\label{E:Cij}
C_{ij}(n,k,l,\alpha,\beta):=U_i\,U_j\,	
		c_{i-k,j-k}(n-k-l,\alpha+2l,\beta+2k),
\end{equation}
notation used being that of \eqref{E:cij} and \eqref{E:U}.
\end{thm}
Observe that the quantities $C_{ij}$ can be put in a square  array (see~Table~\ref{Tab:C}).
\begin{table}[htb]
\caption{The $C$-table \label{Tab:C}}
\normalsize
\[
\begin{array}{cccccc}
    &0&0&\ldots&0& \\[2ex]
    0& C_{kk}&C_{k,k+1}&\ldots&C_{k,n-l}&0     \\[2ex]
    0& C_{k+1,k}&C_{k+1,k+1}&\ldots&C_{k+1,n-l}&0  \\[2ex]
    \multicolumn{6}{c}{\dotfill}\\[2ex]
    0& C_{n-l,k}&C_{n-l,k+1}&\ldots&C_{n-l,n-l}&0  \\[2ex]
     &0&0&\ldots&0& 
\end{array}
\]	
\end{table}

Combining \eqref{E:Cij} with Theorem~\ref{T:c-rec}, we obtain the following theorem.
\begin{thm}
	\label{T:C-rec}
Quantities $C_{ij}\equiv C_{ij}(n,k,l,\alpha,\beta)$ defined in \eqref{E:Cij} 
satisfy the recurrence relation 
	\begin{align}\label{E:C-rec}
C_{i+1,j}=&\frac{1}{A^\ast(i)}\Big\{(i-j)(2i+2j-2n-\alpha+\beta)\,C_{ij}\nonumber \\
&\hphantom{\frac{1}{A^\ast(i)}\Big\{}
+B^\ast(j)\,C_{i,j-1}+A^\ast(j)C_{i,j+1}-B^\ast(i)C_{i-1,j}\Big\},
	 \end{align}
where $k\le i\le n-l-1$, $k\le j\le n-l$ and
\[  
\begin{array}{l}
A^\ast(u):=(u-n)(u-k+1)(u+k+\beta+1)/(u+1),\\[2ex]	
B^\ast(u):=u(u-n-l-\alpha-1)(u-n+l-1)/(u-n-1).
\end{array}  
\] 
We adopt the convention that
$C_{ij}:=0$ if $i<k$, or $i>n-l$, or $j<k$, or $j>n-l$.
The starting values are
\begin{align}	\label{E:Crec-start}
C_{kj}
 =&\,W\cdot
 \frac{(-1)^{j} \,U_j\,}{(\alpha+2l+1)_{n-l-j}(k+\beta+2)_j}\qquad(j=k,k+1,\ldots, n-l), 	
\end{align}
where we use notation \eqref{E:U}, and
\begin{equation}
	\label{E:W}
W:=\rbinom{n}{k}\,
 \frac{(-1)^{k}(\sigma+2k+2l+1)_{n-k-l}(k+\beta+2)_{n-l}}{B(\alpha+2l+1,\,\beta+2k+1)\,(n-k-l)!}.	
\end{equation}
\end{thm}
Using equation \eqref{E:C-rec}, one may obtain the element $C_{i+1,j}$ in terms 
of four elements from the rows number $i$ and $i-1$ (see Table~\ref{Tab:C-cross}).
Now,   the $C$-table can be completed very easily in the following way.
\begin{table}[h!]
\caption{The cross rule \label{Tab:C-cross}}
\vspace*{-3ex}
\normalsize
\[
\begin{array}{ccc}
    &C_{i-1,j}&\\[2ex]
    C_{i,j-1}& C_{ij}&C_{i,j+1}\\[2ex]
    &\fbox{$C_{i+1,j}$}&\\ 
\end{array}
\]	
\end{table}
\newpage
\begin{alg}[Computing the B\'ezier coefficients of the polynomials $ D^{(n,k,l)}_i$]
        \label{A:Ctabcomp}
\ \begin{enumerate}
\item Compute quantities $C_{kk},\;C_{k,k+1},\ldots,\,C_{k,n-l},$
filling the first row of the $C$-table, by the formulas
\begin{align}
	\label{E:Cknl}
	  C_{k,n-l}:=&\rbinom{n}{k}\,\rbinom{n}{l}\;
	  \frac{(-1)^{n-k-l}(\sigma+2k+2l+1)_{n-k-l}}
	                                  {B(\alpha+2l+1,\,\beta+2k+1)\,(n-k-l)!},\\[2ex]
	\label{E:Ckj}
	  C_{kj}:=&\frac{(j-n)(j-k+1)(j+\beta+k+2)}{(j+1)(j-n+l)(j-\alpha-l-n)}\,C_{k,j+1}\nonumber \\[1ex]
	  &\hphantom{\frac{(j-n)(j-k+1)(j+\beta+k+2)}{(j+1)(j-n+l)(j-\alpha-l-n)}} 
	  (j=n-l-1,n-l-2,\ldots,k).
\end{align}
\item For $i=k,k+1,\ldots,n-l-1$ and $j=k,k+1,\ldots,n-l$, compute
      			$C_{i+1,j}$, using the recurrence   \eqref{E:C-rec}.
\end{enumerate}
\end{alg}
\begin{rem}
	\label{R:unconstr}
Obviously, in case $k=l=0$, the above scheme allows to compute the B\'ezier coefficients
$c_{ij}(n,\alpha,\beta)$ of the unconstrained dual polynomial $D^n_{i}(x;\alpha,\beta)$ (cf \eqref{E:DinB}).
\end{rem}

\begin{rem}\label{R:C-symmetry}
If $\alpha=\beta$ and $k=l$, we have the following symmetry property of the $C$-table:
\[  
	\label{E:C-symmetry}
	C_{ij}=C_{n-i,n-j}\qquad(k\le i,j\le n-k). 
\]  
This can be verified easily, using the definition \eqref{E:Cij} of $C_{ij}$ as well as
the identity 
\[Q_m(n-i;\alpha,\alpha,n)=(-1)^mQ_m(i;\alpha,\alpha,n).\] 
Thus, the cost of computing the $C$-table can be halved in this case. 
\end{rem}

 \begin{rem}\label{R:J&RA-form}
J\"uttler \cite{Jue98} studied the case of $k=l$ and
$\alpha=\beta=0$ and
gave a rather complex recurrence relation in $i$ and $k$,  for the coefficients $c^{(k)}_{i,j}$ in 
\[
	D^{(n,k,k)}_i(x;0,0)=\sum_{j=k}^{n-k}c^{(k)}_{i,j}\,B^n_j(x)
	\qquad (k\le i\le n-k).
\]
Rababah and Al-Natour \cite{RA08}  extended his approach to the case of
arbtitrary $\alpha,\;\beta>-1$.
\end{rem}

\section{Applications}
							\label{S:appl}

Consider the following approximation problem.
Given the function $f$ defined on $[0,\,1]$,
 we look for a polynomial $P_m\in\Pi_m$,
\begin{equation}\label{E:Pm}
P_m(t):=\sum_{i=0}^{m}p_i\,B^m_i(t),	
\end{equation} 
which gives minimum value of the squared norm
\begin{equation}\label{E:err}
\| f-P_m\|^2_{L_2}=\langle f-P_m,\,f-P_m\rangle
\end{equation}
with the constraints
\begin{equation}
	\label{E:constr}
\left.\begin{array}{ll}
  f^{(i)}(0)=P^{(i)}_m(0)&\quad  (i=0,1,\ldots,k-1),\\[1.5ex]
  f^{(j)}(1)=P^{(j)}_m(1)&\quad (j=0,1,\ldots,l-1),	
\end{array}\ \right\}
\end{equation}
where $k+l\leq m$.

This general model contains, for instance, 
\begin{enumerate}
\itemsep4pt
\item \textit{multi-degree reduction problem}, where $f=L_n$ is the polynomial of degree $n$,
$n>m$, given in B\'ezier form,
\[
	   L_n=\sum_{i=0}^{n}l_i\,B^n_i
\]
(see, e.g., \cite{WL09} and references given therein);

\item  \textit{computing roots of polynomials by the method of clipping} \cite{SN90,BJ07,LZLW09},
       which uses a specific variant of the multi-degree reduction with low
       value of the degree $m$, say, less than five; 
\item \textit{polynomial approximation of the rational B\'ezier form}, where $f=R_n$,
\[
	R_n=\frac{ \sum_{i=0}^n \omega_ir_i\,B^n_i}
	{ \sum_{i=0}^n \omega_i\,B^n_i}
\]	
(see, e.g., \cite{LWK10} and references given therein).
\end{enumerate}

\begin{thm}
	\label{T:main}
The coefficients $p_0,p_1, \ldots,p_m$ of the polynomial
\eqref{E:Pm} minimising the error \eqref{E:err} with constraints \eqref{E:constr}
are given by
\begin{align}
	\label{E:pi-begin}
	  p_i=&\frac{(m-i)!}{m!} f^{(i)}(0)-\sum_{j=0}^{i-1}(-1)^{i+j}\binom{i}{j}p_j
	  \qquad (i=0,1,\ldots,k-1); \\
	\label{E:pi-end}
	  p_{m-i}=&(-1)^i\frac{(m-i)!}{m!} f^{(i)}(1)
	          -\sum_{j=1}^{i}(-1)^{j}\binom{i}{j}p_{m-i+j}
		   \qquad (i=0,1,\ldots,l-1);\\
	 \label{E:pi-mid}
	  p_{i}=&\sum_{j=k}^{m-l}C_{ij}(m,k,l,\alpha,\beta)\,\langle f,\,B^m_j\rangle 
	  -
	  \left(\sum_{j=0}^{k-1}+\sum_{j=m-l+1}^{m}\right)p_j\,K_{ij}\nonumber \\
 &\hphantom{ p_{i}=\sum_{j=k}^{m-l}c_{ij}(m,k,l,\alpha,\beta)\,\langle f,\,B^m_j\rangle }
	  \qquad\qquad(i=k,k+1,\ldots,m-l), 	
\end{align}
where  $C_{ij}(m,k,l,\alpha,\beta)$ are introduced in \eqref{E:constrDinB} and
\begin{align}		
	\label{E:K}
	K_{ij}	
	=&\binom{m}{j}\rbinom{m}{i}\,\frac{(-1)^{i-k}(k-j)_{m-k-l+1}}{(i-j)(i-k)!(m-l-i)!}
      	\nonumber \\[1ex]
	&\times
	\frac{(\alpha+l+1)_{m-j}(\beta+k+1)_{j}}{(\alpha+l+1)_{m-i}(\beta+k+1)_{i}}.	 
\end{align}
\end{thm}
\begin{proof} Recall that  (see, e.g., \cite[p. 49]{Far96})
\begin{align*}
P_m^{(j)}(0)=&
\frac{m!}{(m-j)!}\sum_{h=0}^{j}(-1)^{j+h}\binom{j}{h}p_{h},\\
P^{(j)}_m(1)=&
\frac{m!}{(m-j)!}\sum_{h=0}^{j}(-1)^{j+h}\binom{j}{h}p_{m-j+h}. 
\end{align*}
Using the above equations 
in \eqref{E:constr}, we obtain the forms 
\eqref{E:pi-begin} and
\eqref{E:pi-end} for the coefficients $p_0,\,p_1,\ldots,p_{k-1}$ and
$p_{m-l+1},\ldots,p_{m-1},p_m$, respectively.

The remaining coefficients $p_i$ are to be determined so that
\[
\| f-P_m\|^2_{L_2}=\| W-\sum_{i=k}^{m-l}p_iB^m_i\|^2_{L_2}	
\]
has the least value, where
\[
W:=f-\left(\sum_{j=0}^{k-1}+\sum_{j=m-l+1}^{m}\right)p_jB^m_j.
\]

Remembering that  $B^m_i$ and $D^{(m,k,l)}_i$ ($k\le i\le m-l$) are dual
bases in the space $\Pi_m^{(k,l)}$, we obtain the formula
\[
p_i=\ip{W}{ D^{(m,k,l)}_i}\\
=\sum_{j=k}^{m-l}C_{ij}(m,k,l,\alpha,\beta)\,\langle f,\,B^m_j\rangle
    -\left(\sum_{j=0}^{k-1}+\sum_{j=m-l+1}^{m}\right)p_jK_{ij},  
\] 
where
\[
	K_{ij}:=\ip {B^m_j}{D^{(m,k,l)}_i}.
\]
Formula \eqref{E:K} was obtained in \cite{LWK10}. Equation \eqref{E:pi-mid} now readily 
follows.
\end{proof}

\begin{rem}\label{R:1}
Evaluation of quantities $\langle f,\,B^m_j\rangle$, which are present in equation
\eqref{E:pi-mid}, is a simple task in case of multi-degree reduction
or computing zeros by B\'ezier clipping, as  we have 
\[
	\ip{f}{B^m_j}=\ip{L_n}{B^m_j}=\sum_{i=0}^{n}l_i\,
	\ip{B^n_i}{B^m_j},
\]
and 
\[
	\ip{ B^n_i}{B^m_j}= B(\alpha+1,\beta+1)\binom{n}{i}\binom{m}{j}
	\frac{(\alpha+1)_{n+m-i-j}(\beta+1)_{i+j}}{(\alpha+\beta+2)_{n+m}}.
\]
In case of approximation of the rational function, we have
\[
	\ip{f}{B^m_j}=\ip{R_n}{B^m_j}
	= \sum_{i=0}^n \omega_ir_i I_{ij},
\]
where the integrals
\(
	I_{ij}:=\ip {B^n_i}
	{{B^m_j}	
	\Big/{ \sum_{h=0}^n \omega_h\,B^n_h}}
\)
should be computed using numerical methods (see, e.g., \cite{LWK10}).
\end{rem}

\section*{Conclusions}
                    \label{sec:Concl}

We gave explicit formulae   for the B\'ezier coefficients of the constrained 
dual Bernstein basis polynomials
 in terms of the Hahn orthogonal polynomials. Using difference equation satisfied by
the latter polynomials, we obtained a recursive algorithm to compute  efficiently
these  coefficients. 
Applications of this result to several problems is discussed, namely 
to the multi-degree reduction problem for the
B\'ezier curves, computing roots of polynomials by the method of B\'ezier clipping
and  approximation of the rational B\'ezier curves by B\'ezier curves.

\end{document}